\documentclass[12pt]{article}

\usepackage{amsmath,amssymb,amsthm,amsfonts}

\usepackage{times}
\usepackage{enumerate}
\usepackage{picinpar}
\usepackage{layout}
\usepackage{verbatim}
\usepackage{epsfig}
\usepackage{graphicx}

\newtheorem{theorem}{Theorem}[section]
\newtheorem*{theorem*}{Theorem}

\newtheorem{lemma}[theorem]{Lemma}

\newtheorem{assumption}[theorem]{Assumption}

\newtheorem*{corollary*}{Corollary}

\newtheorem*{conjecture*}{Conjecture}

\theoremstyle{remark}

\theoremstyle{definition}

\newtheorem{example}[theorem]{Example}
\numberwithin{figure}{section}

\newcounter{fig}
\newcommand{\ddt}{\frac{d}{dt}}
\newcommand{\lskip}{\phantom{.}}
\newcommand{\var}{\hbox{Var}}

\def\R{\mathbb R}

\def\E{\mathbb E}

\newcommand{\Var}{\hbox{Var}}

\title{Propagation of Fluctuations in Biochemical Systems, II:
  Nonlinear Chains } \author{David F. Anderson$^{1}$ and Jonathan
  C. Mattingly,$^{2}$}

\begin{document}

\maketitle

\footnotetext[1]{Department of Mathematics, University of
  Wisconsin-Madison, Madison, Wi 53706, anderson@math.wisc.edu}
\footnotetext[2]{Department of Mathematics, Duke University, Durham,
  NC 27708}

\begin{abstract}
  We consider biochemical reaction chains and investigate how random
  external fluctuations, as characterized by variance and coefficient
  of variation, propagate down the chains.  We perform such a study
  under the assumption that the number of molecules is high enough so
  that the behavior of the concentrations of the system is well
  approximated by differential equations.  We conclude that the
  variances and coefficients of variation of the fluxes will decrease
  as one moves down the chain and, through an example, show that there
  is no corresponding result for the variances of the chemical
  species.  We also prove that the fluctuations of the fluxes as
  characterized by their time averages decrease down reaction chains.
  The results presented give insight into how biochemical reaction
  systems are buffered against external perturbations solely by their
  underlying graphical structure and point out the benefits of
  studying the out-of-equilibrium dynamics of systems.
\end{abstract}

\section{Introduction}

In \cite{AndThesis} and \cite{AndMattReed06} we began a study of
biochemical reaction systems subjected to random, external forcing.
The question we considered, and continue with here, is the following:
if we add random, external forcing to the input of a biochemical
reaction system, how do those fluctuations (characterized by their
variance and coefficient of variation) propagate through the entire
system?  The broader aims of this paper are to gain a better
understanding of how the network topology of biochemical reaction
systems suppresses or otherwise alters the behavior of fluctuations in
reaction systems and to point out the benefits of studying the
out-of-equilibrium dynamics of systems.

In \cite{AndMattReed06} we studied systems under the two simplifying
assumptions that the kinetics were all mass action and that each
reaction involved turning precisely one substrate into another
substrate.  Therefore, in the terminology of Horn, Jackson, and
Feinberg, each complex consisted of a single species
(\cite{FeinbergLec79}\cite{HornJack72}).  Thus, we allowed reactions
of the form $A \rightarrow B$, but not $A+B \rightarrow C$.  These two
assumptions caused the differential equations governing the
concentrations of the species to be linear and so we referred to them
as linear SSC (single species complex) systems.  Considering linear
SSC systems decreased some of the technical difficulties of the
analysis while still allowing us to probe how different network
structures affect the propagation of fluctuations.  Under these
assumptions we proved that the variances of fluxes decrease down
reaction chains and that side reaction systems and feedback loops
lower the variance of the flux out of reaction chains.  A natural
question is whether or not these results from \cite{AndMattReed06}
hold when we drop one or both of the simplifying assumptions.  The
main purpose of this paper is to demonstrate a biologically
significant result from \cite{AndMattReed06} that does hold when we
drop both the SSC and mass action assumptions: the variances and
coefficients of variation of fluxes decrease as one moves down a
non-reversible reaction chain.

For an example of a reaction chain consider the following biochemical
system: a substrate, $S$, enters the system at a constant rate, $I >
0$.  This substrate then combines with an enzyme, $E$, to form $ES$,
which is then degraded to some product substrate $P$ plus the original
enzyme.  Finally, the product $P$ leaves the system.  If the
concentration of the enzyme is taken to be so large as to be assumed
constant and the reactions are non-reversible then the following graph
faithfully models our system:
\begin{equation}
  \begin{array}{ccccccc}
    I  & & F_1 & & F_2 &  & F_{3}\\
    \longrightarrow & S & \longrightarrow & ES & \longrightarrow & P &
    \longrightarrow,
  \end{array}
  \label{reaction_example}
\end{equation}
where $F_1$, $F_2$ and $F_3$ are functions that give the rates of the
respective reactions.  Let $s,es,$ and $p$ be the concentrations of
$S$,$ES$, and $P$, respectively.  Then, if the kinetic functions
$F_1,$ $F_2,$ and $F_3$ are functions of the reactant substrates only,
the differential equations governing the temporal evolution of the
concentrations are
\begin{align}
  \begin{split}
    \dot s(t) &= I - F_1(s(t))\\
    \dot {es}(t) &= F_1(s(t)) - F_2(es(t))\\
    \dot p(t) &= F_2(es(t)) - F_3(p(t)).
  \end{split}
  \label{ex:model}
\end{align}
If the functions $F_i$ are differentiable, monotone increasing, and
satisfy $F_i(0) < I < \lim_{x \to \infty} F_i(x)$, then it is easily
seen that, independent of initial conditions, the system
\eqref{ex:model} will converge to the steady state $(\bar s, \bar
{es}, \bar p) = (F_1^{-1}(I), F_2^{-1}(I), F_3^{-1}(I))$.  However, if
the input to the system \eqref{reaction_example} is allowed to
fluctuate in time, then each concentration will fluctuate, and, hence,
each flux, $F_{i}$, will also fluctuate.  If the fluctuations are
random, we can ask what the variance or coefficient of variation of
each flux is with respect to that randomness, and how they relate.  It
is the goal of this paper to prove
\begin{align}
  \Var(F_1(s(t))) &>  \Var(F_2(es(t))) > \Var(F_3(p(t)))  \label{align:ineq1}\\
  CV(F_1(s(t))) &> CV(F_2(es(t))) > CV(F_3(p(t))),
  \label{align:ineq2}
\end{align}
where Var$(\cdot)$ and $CV(\cdot)$ represents variance and coefficient
of variation, respectively. (Notice that the inequalities are strict.)

The reaction chain given in \eqref{reaction_example} is an example of
an SSC chain because each node of the network graph consists of a
single substrate.  In general, a reaction chain is any biochemical
system of the following form:
\begin{equation}
  \begin{array}{ccccccccccc}
    I  & & F_1 & & F_2 &  & F_{n-1} & & F_n \\
    \longrightarrow & C_1 & \longrightarrow & C_2 & \longrightarrow &
    \dots & \longrightarrow & C_n & \longrightarrow,
  \end{array}
  \label{reactionchain}
\end{equation}
where $I \in \R_{> 0}$ is the constant input to the system, the
complexes, $C_i$, are linear combinations of the substrates, and $F_i:
\R^{m_i}_{\ge 0} \to \R_{\ge 0}$ are the reaction kinetics (where
$m_i$ is the number of distinct substrates composing complex $C_i$).
In \cite{AndMattReed06} we showed that if the constant input $I$ is
replaced by the fluctuating in time random process $I +
\xi(t,\omega)$, where $\xi(t,\omega)$ is either white noise or a mean
zero, finite variance, stationary stochastic process such that
$\xi(t,\omega) \ge -I$, and if the system \eqref{reactionchain} is a
linear SSC chain, then for all $i \ge 1$, $\var(F_{i}) >
\var(F_{i+1})$, where the variance is computed according to the unique
stationary measure to which the distribution of the species converges.
In this paper, we prove that this result still holds when we drop the
assumption that the kinetics are mass action and the assumption that
each complex consists of a single species.  The main assumption on the
kinetics will be that they are monotone increasing in each of their
dependent variables (so, for example, we may consider Michaelis-Menten
kinetics).  In dropping the SSC assumption we will show that the
result still holds when the complexes are composed of multiple species
so long as each species appears in precisely one complex.  Throughout,
we will refer to systems for which complexes can by composed of
multiple species yet each species appears in a single complex as MSC
(multiple species complexes) systems.

The goal of this paper is to prove that variances of fluxes caused by
an external stochastic input decrease down a non-reversible reaction
chain. We will show that if $I$ is the average input to a reaction
chain, then (once the system has reached its statistical equilibrium)
the mean of each flux is also equal to $I$.  Therefore, saying that
the variances of the fluxes decrease down a reaction chain is
equivalent to saying that the coefficients of variation of the fluxes
decrease down a reaction chain.  That is, equation \eqref{align:ineq1}
is equivalent to equation \eqref{align:ineq2}.  Because of this
equivalence between the magnitudes of variances and the magnitudes of
coefficients of variation, each result in this paper is stated in
terms of variance alone and it is understood that each result is still
valid if $\Var(\cdot)$ is replaced with $CV(\cdot)$.

Throughout, we allow external perturbations to be white noise
processes or mean zero, finite variance, stationary stochastic
processes.  Considering white noise processes is useful because it
allows one to make use of the It\^o Calculus with which stronger
results (i.e. fewer restrictions on the rate functions, $F_i$) can be
proven than if you solely consider arbitrary perturbations.  Also, if
the input flux to a reaction system is perturbed by white noise, then
all other fluxes are perturbed by mean zero, finite variance,
stationary continuous processes.  Therefore, one may construct
continuous, stationary perturbations from white noise processes by
allowing a pseudo-species to be perturbed by white noise and
considering the output from the pseudo-species as the input to the
reaction system of interest.  In doing so, one is allowed to use the
stronger white noise results as opposed to the stationary noise
results.  Thus, allowing both types of perturbations is quite natural.
However, we point out that we do not feel the choice of external
forcing is critical because the broader aim of this paper is to study
the out-of-equilibrium dynamics of biochemical systems and both
choices of perturbation achieve this aim.

The layout of the paper is as follows.  In Section \ref{sec:SSCchains}
we consider SSC chains.  In Section \ref{sec:MSCchains} we consider
MSC chains.  Complementing the main results are two important
examples.  In Section \ref{sec:SSCchains}, Example \ref{example1} is a
nonlinear chain perturbed by white noise for which the variance (and
CV) of the species increase down the chain.  Hence, there is no
corresponding ``decreasing fluctuation'' result for the species of
reaction chains.  In Section \ref{sec:MSCchains}, Example
\ref{example2} demonstrates that the assumption that each species is
in precisely one complex is a necessary one.  In both examples we use
a Monte Carlo simulation to arrive at our conclusions.  The proofs of
all the results in this paper are found in Appendix \ref{appA}.

This paper is part of a larger research project in which the main
biological goal is to understand how network topology affects how
biochemical systems react to large-scale, random perturbations to
their inputs.  There are two distinct approaches we take in trying to
achieve this goal.  In the first, we apply random fluctuations to {\it
  in silico} representations of specific biological systems.  We can
then identify reactions, substrates, or whole subsystems that are
buffered against the fluctuations, i.e. are homeostatic.  We can then
take the system apart piece by piece through {\it in silico}
experimentation to discover the regulatory mechanisms that give rise
to the homeostasis.  In the second, we prove theorems about how random
fluctuations propagate through relatively simple, but biologically
relevant, systems.  We are interested in how these systems magnify or
suppress fluctuations as this may give clues as to why these systems
are structured as they are.  In this second approach it is the
out-of-equilibrium dynamics that is being probed in order to give
information on the emergent properties of the system.  This paper,
like \cite{AndThesis} and \cite{AndMattReed06}, takes the second
approach; for an example of the first, see \cite{nijhout06}.

Due to the inherent randomness in the making and breaking of chemical
bonds, biochemical reaction systems are, at their most fundamental
level, modeled as jump Markov processes
(\cite{Delbruck40}\cite{Gill76}\cite{OthGadLee05}\cite{ThatOud01}\cite{GansKrieger1960}).
However, if one scales up the volume and number of molecules in a
system while keeping the initial concentrations constant, then this
intrinsic randomness becomes negligible at the scale of
concentrations.  One is then able to faithfully model the
concentrations of the substrates by a system of differential equations
(\cite{Kurtz72}).  As in \cite{AndThesis} and \cite{AndMattReed06} we
consider systems in this scaling limit.  Thus, the random external
forcing in this paper is on the scale of concentrations (and not of
individual molecules) and the concentrations of the species are
modeled by differential equations and not by discrete jump processes.
For a more detailed comparison between the randomness in this paper
and the inherent randomness of biochemical systems, see
\cite{AndMattReed06}.

\section{SSC chains with random perturbations}
\label{sec:SSCchains} 

In this section we consider nonlinear SSC chains subjected to random
perturbations.

\subsection{The model}
\label{model}

A non-reversible SSC chain with a constant input is a biochemical
reaction system with the following graphical structure:
\begin{equation}
  \begin{array}{ccccccccccc}
    I  & & F_1 & & F_2 &  & F_{n-1} & & F_n \\
    \longrightarrow & X_1 & \longrightarrow & X_2 & \longrightarrow &
    \dots & \longrightarrow & X_n & \longrightarrow,
  \end{array}
  \label{sscchain}
\end{equation}
where $I > 0$ is the constant input to the system, $X_i$ are the
species (and complexes) of the system, and $F_i: \R_{\ge 0} \to
\R_{\ge 0}$ are the reaction kinetics.  If we let $\{x_i\}$ denote the
concentrations of the species $\{X_i\}$, then the temporal evolution
of $x(t)$ is governed by the following differential equation:
\begin{align}
  \begin{split}
    \dot x_1 &= I-F_1(x_1)\\
    \dot x_2 &= F_1(x_1) - F_2(x_2)\\
    &\ \ \vdots \\
    \dot x_n &= F_{n-1}(x_{n-1}) - F_n(x_n).
  \end{split}
  \label{ssc}
\end{align}

In the sequel we make the following standing assumptions on the
functions $F_i$.

\begin{assumption}Each $F_i$ is a real valued $C^1$ function of
  $[0,\infty)$ with the following properties:
  \begin{enumerate}[a)]
  \item $F_i(0) = 0$.
  \item For all $x \in \mathbb{R}_{> 0}$, $F_i'(x) > 0$ .
  \item $\displaystyle{\lim_{x \rightarrow \infty}}F_i(x) > I$.
    \label{nobuildup}
  \end{enumerate}
  \label{ssc_assump}
\end{assumption}
Note that condition \ref{nobuildup}) guarantees that mass will not
build up at any point along the chain so long as the input is kept at
the constant value $I$.  This assumption is also reasonable for
systems for which the input is being perturbed by a mean zero random
process and will be used to keep concentrations from escaping to
infinity.

We will consider two different classes of random perturbations of the
input $I$. The first will be white in time while the second will be
almost surely continuous in time.  Since the kinetics, $F_i$, are
defined only on the positive portion of the real line, it is important
that any noise we consider as a perturbation to the input will never
drive the concentrations of the species into the negative portion of
the real line. Hence we will impose restrictions on the perturbations
to ensure that the specie concentrations stay non-negative at all
times.  Because we consider two different classes of perturbations of
the input, we consider two different mechanisms to achieve this goal.

In the case of the white in time random perturbation, we multiply the
noise term, $dB_t$, by a function, $\theta_{\delta}(\cdot)$, that
turns the noise off if $x_1$ approaches zero.  This property of
$\theta_{\delta}$ combined with the dynamics governing the
concentration of $X_1$, ensures that $x_1$ remains non-negative for
all time which, in turn, keeps all other concentrations non-negative
for all time.  The concentration of $x_1$ is now governed by a
stochastic differential equation, while the equations for the
remaining $x_i$ stay as in \eqref{ssc}.  That is,
\begin{align}
 \begin{split}
     dx_1 &= (I-F_1(x_1))dt + \sigma \theta_{\delta}(x_1) dB(t)\\
     \dot x_2 &= F_1(x_1) - F_2(x_2)\\
    &\ \ \vdots \\
    \dot x_n &= F_{n-1}(x_{n-1}) - F_n(x_n),
  \end{split}
     \label{sscsdes}
\end{align}
where $B(t) = B(t,\omega)$ is standard one dimensional Brownian
motion, $\sigma \in \R_{>0}$, and for some small $\delta > 0$,
$\theta_{\delta}(x)=1$ for all $x > \delta$, $\theta_{\delta}(0) = 0$
and $\theta_{\delta}(x)$ is $C^{\infty}$ and monotone increasing.
Since $F_1(0) = \theta_{\delta}(0) = 0$ and $I > 0$, $x_1(t) \ge 0$
for all $t > 0$ if $x_1(0) > 0$.

The second class of perturbations to the input considered in this
paper are mean zero, finite variance, stationary, random processes,
$\xi(t,\omega)$, that are continuous in time at almost every moment of
time.  To guarantee that the concentrations of the chemical substrates
remain non-negative for all time we only consider perturbations such
that $\xi(t,\omega) \ge -I$ for all $t$ and $\omega$ (and so we no
longer need the function $\theta_{\delta}(\cdot)$ used in the white in
time setting).  In order to keep the reaction system away from
equilibrium, we make the added restriction that for each choice of
$\omega$, $\xi(t,\omega)$ is non-constant on all time intervals larger
than some fixed value $a = a(\omega)$.  We will typically write
$\xi(t)$ instead of $\xi(t,\omega)$. The almost everywhere continuity
of $\xi(t)$ allows the possibility of isolated jumps and allows us to
use a standard differential equation for $x_1$ (in contrast to the
It\^o stochastic differential equation used in \eqref{sscsdes}).  In
this case, the equations governing the behavior of the concentrations
are
\begin{align}
 \begin{split}
   \dot x_1 &= I-F_1(x_1) + \xi(t) \\
   \dot x_2 &= F_1(x_1) - F_2(x_2)\\
   &\ \ \vdots \\
   \dot x_n &= F_{n-1}(x_{n-1}) - F_n(x_n).
  \end{split}
     \label{sscstation}
\end{align}
To prove the existence of a stationary state we will further assume
that the distribution of the noise's future is completely determined
by its past.  An example of such a $\xi(t)$ is a Markov processes
whose future distribution depends only on its present value.

\subsection{Decreasing variance for SSC chains}
\label{results}
We are mainly interested in describing the system once it has settled
into a statistical equilibrium and any behavior that is transient in
time has passed. Such statistical steady states are characterized by a
stationary solution. A solution $x^*(t)$ is stationary if for any
collection of times $t_1 < t_2 < \cdots < t_n$ and any $s$ the
distribution of the vector $(x^*(t_1+s), x^*(t_2+s),\cdots,
x^*(t_n+s))$ is independent of $s$. When the forcing is Brownian as in
\eqref{sscsdes}, the solution is a Markov process and the distribution
of $x^*(t)$ at any time $t$ is an invariant measure for the associated
Markov semigroup. An invariant measure, $\mu$, is a measure on the
state space of the system, $\R^n$, such that if the initial condition
is chosen according to $\mu$ then solutions at any time $t\geq 0$ are
also distributed as $\mu$. More precisely, if for all measurable $A
\subset \R^n$, $P( x(0) \in A) = \mu(A)$ implies that $P( x(t) \in
A)=\mu(A)$ for all $t \geq 0$, then $\mu$ is invariant to the dynamics
of the system.  Therefore, when the forcing is Brownian, a stationary
solution exists.  When the forcing is a stationary process $\xi(t)$ as
in \eqref{sscstation} more care must be taken to obtain a stationary
solution to the dynamics as the solution need not be a Markov process.

We will concern ourselves with the existence and basic properties of
stationary solutions and invariant measures at the end of this
section. First we state the principal result of the article and give a
few numerical examples to illustrate its use.  The following theorem
is proved in Appendix \ref{appA_principal}

\begin{theorem}[Decreasing variance down a nonlinear SSC chain]
  Let $x^*(t)$ be a stationary solution for the dynamics given in
  either equation \eqref{sscsdes} or \eqref{sscstation}. Then for all
  $1 \le i \le n$ and all $t$
  \begin{equation*}
    \var\left(F_{i}(x_{i}^*(t)) \right) > \var\left(F_{i+1}(x_{i+1}^*(t))
    \right)\;.
  \end{equation*}
  \label{nonlineardecrease}
\end{theorem}

We note that the variances of Theorem \ref{nonlineardecrease} are
computed with respect to the choice of randomness, $\omega$, in
$B(t,\omega)$ or $\xi(t,\omega)$.  That is, Theorem
\ref{nonlineardecrease} gives the variance as an average over the
choice of perturbations.  In Subsection \ref{sec:paths} we give a
similar result except the variance is computed as a time average over
a single path.  In many natural settings (including those given in the
examples below), the two notions are equivalent and, therefore, give
the same intuition about fluctuations down reaction chains.

We now give three examples where the preceding theorem holds. For the
moment we will assume that the systems possess a stationary solution
to which the statistics of the solutions converge as $t \rightarrow
\infty$.  At the end of the section we will prove that the preceding
assumption holds for any initial condition.

\begin{example}[Species variances need not decrease]
  Consider the following SSC chain with Michaelis-Menten kinetics
  \begin{equation*}
    \begin{array}{cccccc}
      10 + \sigma \theta_{\delta}(x_1) dB_t & & F_1(x_1) & & F_2(x_2) & \\
      \longrightarrow & X_1 & \longrightarrow & X_2 & \longrightarrow &,
    \end{array}
  \end{equation*}
  where $\sigma = 1$, $\delta = .001$, $F_1(x_1) = x_1$, $F_2(x_2) =
  \frac{12x_2}{1 + x_2}$.  We will see in Theorem \ref{measurelemma}
  that this system possesses a unique invariant measure to which the
  statistics of the trajectories converge.  Using Matlab to perform a
  Monte Carlo simulation we computed the means, variances, and
  coefficients of variation of the species and fluxes to be the
  following:
  \begin{center}
    \begin{tabular}{|c|c|c|c|c|}
      \hline
      & $x_1$ & $x_2$ & $F_1(x_1)$ & $F_2(x_2)$  \\
      \hline
      mean & 10 & 5.18 & 10 & 10 \\
      \hline
      variance & .5 & 1.19 & .5 & .124 \\
      \hline
      CV & .07 & .21 & .109 & .035\\
      \hline
    \end{tabular}
  \end{center}
  As guaranteed by Theorem \ref{nonlineardecrease}, the fluctuations
  of the fluxes decrease down the chain.  However, no matter the
  measure we use (variance or CV), the fluctuations of $x_{2}$ are
  always greater than those of $x_1$.  Therefore, there is no
  counterpart to Theorem \ref{nonlineardecrease} pertaining to the
  species of an SSC reaction chain.

  To understand why the fluctuations of $x_2$ are higher than those of
  $x_1$, consider the plot of $F_2(x) = 12x/(1+x)$ in Figure
  \ref{M-M}.  The horizontal lines at $x_1 = 9,10,$ and $11$ represent
  possible fluxes into species $X_2$, while the vertical lines show
  what the equilibrium value of $x_2$ would be corresponding to that
  input.  While the perturbed system will never settle to an
  equilibrium, the kinetics will always be driving the concentration
  of $X_2$ towards the solution of $\frac{12x_2(t)}{1+x_2(t)} =
  x_1(t)$.  Therefore, minor fluctuations in the input to the species
  $X_2$ give rise to large fluctuations in $x_2$.

  \begin{figure}
    \begin{center}
      \includegraphics[height=2.2in]{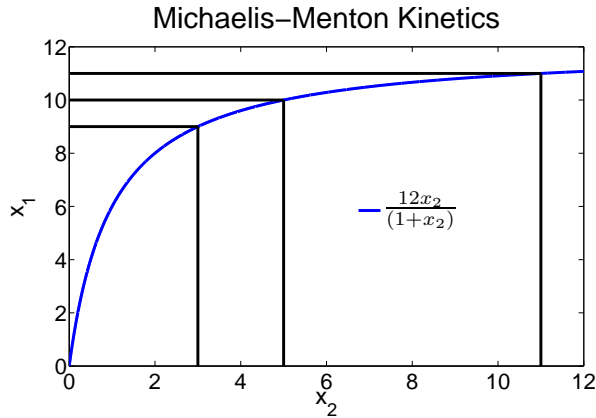}
    \end{center}
    \caption{The horizontal lines represent possible inputs to species
      $X_2$ from species $X_1$ and the vertical lines represent the
      value of $x_2$ that would give an equilibrium to the system for
      a given input.  We therefore see that minor fluctuations in
      $F_1(x_1) = x_1$ can correspond with large fluctuations in
      $x_2$.}
    \label{M-M}
  \end{figure}
\label{example1}
\end{example}

\begin{example}[Continuous input and unbounded kinetics]
  Consider the following SSC chain
  \begin{equation*}
    \begin{array}{cccccc}
      10 + \xi(t) & & F_1(x_1) & & F_2(x_2) &  \\
      \longrightarrow & X_1 & \longrightarrow & X_2 & \longrightarrow ,
    \end{array}
  \end{equation*}
  where $-10 \le \xi(t)$ is a modified Ornstein-Uhlenbeck process
  defined in Appendix \ref{appB}, $F_1(x_1) = x_1^2$ and $F_2(x_2) =
  x_2^2/(1+x_2)$.  Because both $F_1(x)$ and $F_2(x)$ are unbounded as
  $x \to \infty$, we will see in Theorem \ref{cor:arb_lineargrowth}
  that the system possesses a unique stationary solution to which the
  statistics of the trajectories converge.  Using Matlab to perform a
  Monte Carlo simulation we computed the means, variances, and
  coefficient of variation of the fluxes to be the following:
  \begin{center}
    \begin{tabular}{|c|c|c|c|}
      \hline
      & $10 + \xi(t)$ & $F_1(x_1)$ & $F_2(x_2)$  \\
      \hline
      mean & 10 & 10 & 10 \\
      \hline
      variance & 8 & 6.8 & 3.9  \\
      \hline
      CV & 0.28 & .26 & .20\\
      \hline
    \end{tabular}
  \end{center}
  The variances and coefficients of variation of the fluxes decrease
  down the chain, as guaranteed by Theorem \ref{nonlineardecrease}.
\label{4-1example}
\end{example}

\begin{example}[Continuous input and Michaelis-Menten kinetics]
  Consider the following SSC chain
  \begin{equation*}
    \begin{array}{cccccc}
      4 + \xi(t) & & F_1(x_1) & & F_2(x_2) & \\
      \longrightarrow & X_1 & \longrightarrow & X_2 & \longrightarrow,
    \end{array}
  \end{equation*}
  where $-4 \le \xi(t) \le 4$ is a modified Ornstein-Uhlenbeck process
  defined in Appendix \ref{appB}, $F_1(x_1) = 11x_1/(1+x_1)$, and
  $F_2(x_2) = 10x_2/(1+x_2)$.  We will see in Theorem
  \ref{cor:arb_bdd} that this system possesses a unique stationary
  solution to which the statistics of the trajectories converge. Using
  Matlab to perform a Monte Carlo simulation we computed the means,
  variances, and coefficient of variation of the fluxes to be the
  following:
  \begin{center}
    \begin{tabular}{|c|c|c|c|}
      \hline
      & $4 + \xi(t)$ & $F_1(x_1)$ & $F_2(x_2)$  \\
      \hline
      mean & 4 & 4 & 4 \\
      \hline
      variance & 4.2 & 3.3 & 2.9  \\
      \hline
      CV & .51 & .46 & .43\\
      \hline
    \end{tabular}
  \end{center}
  The variances and coefficients of variation of the fluxes decrease
  down the chain, as guaranteed by Theorem \ref{nonlineardecrease}.
  \label{4-2example}
\end{example}


We now show the random dynamics given in equation \eqref{sscsdes}
possesses a unique invariant measure. This invariant measure generates
a stationary solution $x^*(t)$ when extended to paths from $-\infty$
to $\infty$. Similarly, under some additional assumptions on $\xi$ and
the $F_i$'s, we show that the dynamics given in equation
\eqref{sscstation} possesses a unique stationary solution $x^*(t)$
(the concept of an invariant measure does not directly make sense for
\eqref{sscstation} since the dynamics are not necessarily Markovian,
see Appendix \ref{assumptionXi}). In addition, in both settings we
show that the statistics of the trajectories converge to those of the
stationary solution $x^*(t)$ as $t \to \infty$. In other words, for
any $x(0)$ and measurable $A \subset \R^n$,
\begin{align}
  P( x(t) \in A) \rightarrow P( x^*(t) \in A)= P( x^*(0) \in A), \text{
    as $t \rightarrow \infty$.}
\end{align}
This means that the long time statistics of the solutions are
independent of the initial condition and the result on the decrease of
fluctuations is applicable on long time intervals.  Of course in the
setting of \eqref{sscsdes}, $x^*(t)$ is distributed as the invariant
measure $\mu$ so $ P( x^*(0) \in A)=\mu(A)$. In the setting of
\eqref{sscstation}, $x^*(t,\xi)$ should be viewed as a function of the
entire past of the noise.

The next three results apply, respectively, in the three preceding
examples to ensure the existence of an unique stationary solution to
whose statistics the statistics of arbitrary trajectories converge in time.
The first result covers the case of white in time forcing while the
second two apply to stationary forcing. The proofs of all three are
contained in Appendix \ref{appA_measures}.

\begin{theorem}[Ergodicity of the SSC chain with white noise]
  Equation \eqref{sscsdes}  possesses a unique invariant measure, $\mu$, on
  $\mathbb R^n$.  Furthermore, the distribution of any solution to
  equation \eqref{sscsdes} converges to $\mu$ as $t \to \infty$.
  \label{measurelemma}
\end{theorem}

To prove the existence of a stationary solution to equation
\eqref{sscstation}, we need to assume that the distribution of the
future of the noise, $\xi(t)$, is determined by its past (such as for
Markov processes). This is made precise in Appendix
\ref{assumptionXi}.  We also need additional assumptions on the
functions $F_i$. We give two versions of these assumptions.

\begin{theorem} Let $\xi$ be as in Appendix \ref{assumptionXi}. Under
  the additional assumption that the rate functions $F_i$ are
  unbounded as $x \to \infty$, equation \eqref{sscstation} possesses a
  unique stationary solution, $x^*(t)$.  Furthermore, any solution
  $x(t)$ to equation \eqref{sscstation} converges to $x^*(t)$ as $t
  \to \infty$.
  \label{cor:arb_lineargrowth}
\end{theorem}

In the event that any of the $F_i$ are bounded, we need a bound on the size of
$\xi(t)$.
\begin{theorem} Let $\xi$ be as in Appendix \ref{assumptionXi}. Define
  $K = \displaystyle{\min_{i}\{\lim_{x \to \infty}F_i(x) - I\}}$.
  Under the additional assumption that $-I \le \xi(t) \le M < K$, for
  all $t$ and some $M < K$, equation \eqref{sscstation} possesses a
  unique stationary solution, $x^*(t)$.  Furthermore, any solution
  $x(t)$ to equation \eqref{sscstation}, converges to $x^*(t)$ as $t
  \to \infty$.
\label{cor:arb_bdd}
\end{theorem}

In the white in time setting, the system is in fact ergodic and hence
by Birkhoff's ergodic theorem we know that for almost every
realization the time average of any statistic converges to the value
of the statistic in the invariant measure. Combining this with the
strong mixing properties of such a system we have that
\begin{equation*}
  \lim_{t \to \infty} \frac{1}{t}\int_0^t
  (F_i(x_i(s)) - I)^2 ds=  \var\left(F_{i}(x_{i}^*(t)) \right)
\end{equation*}
for almost every realization of the Brownian forcing and every initial
condition $x_0$.  For this to hold in the setting of
\eqref{sscstation}, we need to assume in addition that the stationary
measure on $\xi$ is ergodic. Even without this assumption, the next
section shows that one can say something in general. This underlines
the fact that the decrease of variance is really a pathwise phenomenon
due to the dynamics.

\subsection{Pathwise perturbations}
\label{sec:paths}

The variance described in the previous sub-section is computed with
respect to the probability measure of the perturbations.  More
precisely, if $\omega$ is the realization of the perturbation then
$Var(F_i(x(t,\omega))) = \E_{\omega}(F(x_i(t,\omega)) - I)^2$, i.e. is
an average over the realizations of the noise. Another natural way to
characterize how perturbations propagate down chains is to consider
the time averages of paths.  Consider again the dynamics given by
equation \eqref{sscstation}, except now the only assumptions on
$\xi(t)$ are pathwise assumptions:
\begin{equation}
  i.) \lim_{t \to \infty} \frac{1}{t}\int_0^t \xi(s) ds = 0 \ \ \hbox{ and
  } \ \ ii.)  \limsup_{t \to \infty} \frac{1}{t}\int_0^t
  \xi(s)^2 ds < \infty. 
  \label{zeroavg}
\end{equation}
That is, we now assume that the time average for $\xi(t)$ converges to
zero and that the time average of the square is bounded above.  The
following theorem states that the pathwise variances of the fluxes
do not increase  down reaction chains and is proved in Appendix \ref{appA_avg}.

\begin{theorem}
  Consider equation \eqref{sscstation} where $\xi(t)$ satisfies
  \eqref{zeroavg}.  Then for all $i \ge 1$, the following hold:
  \begin{enumerate}
  \item $\displaystyle \lim_{t \to \infty} \frac{1}{t}\int_0^t
    F_i(x_i(s))ds = I$.
  \item $\displaystyle\liminf_{t \to \infty} \left(\frac{1}{t}\int_0^t
      \xi(s)^2 ds - \frac{1}{t}\int_0^t (F_i(x_i(s)) - I)^2 ds \right)
    \ge 0$.
  \item $\displaystyle \liminf_{t \to \infty} \left(\frac{1}{t}\int_0^t
      (F_i(x_i(s)) - I)^2 ds - \frac{1}{t}\int_0^t
      (F_{i+1}(x_{i+1}(s)) - I)^2 ds \right) \ge 0$.
  \end{enumerate}
  \label{thm:avg}
\end{theorem}

\section{MSC chains with random perturbations}
\label{sec:MSCchains}

We now consider MSC chains with random perturbations.  We will again
allow perturbations that are white in time or that are stationary,
mean zero, finite variance and continuous for almost every $t$ and
that satisfy the conditions of Appendix \ref{assumptionXi}.  Consider
a reaction chain, \eqref{reactionchain}, where each complex, $C_i$,
consists of $m_i$ unique species and no species is contained in more
than one complex.  Thus, if $x(t)$ is the vector representing the
species concentrations at time $t$, then $x(t) \in \R^{m_1 + \cdots +
  m_n}$.  Let $X_{i}^{j}$ represent the $j$th species in complex $i$
and $v_{ij}$ be the multiplicity of species $X_{i}^{j}$ in Complex
$i$.  For example, if the reaction chain is
\begin{equation*}
  \longrightarrow X_{1}^{1} + 2X_{1}^{2} \longrightarrow 3X_{2}^{1}
  \longrightarrow,
\end{equation*}
then $v_{11} = 1, v_{12} = 2,$ and $v_{21} = 3$.  

If $F_i$ represents the reaction rate from complex $C_i$ to complex
$C_{i+1}$ we have that,
\begin{equation*}
  F_i(x(t)) = F_i(x_i^1,\dots,x_i^{m_i}):\R^{m_i} \to \R.
\end{equation*}
We assume each $F_i$ satisfies Assumption \ref{msc_assump} which is
analogous to Assumption \ref{ssc_assump}.
\begin{assumption} $F_i$ is a real valued $C^1$ function of
  $[0,\infty)^{m_i}$ with the following properties:
  \begin{enumerate}[a)]
  \item If $x_i^j = 0$ for any $1 \le j \le m_i$, then $F_i(x) = 0$.
  \item If $1 \le j \le m_i$ and $x \in \R^{m_i}_{>0}$, then
    $\frac{d}{dx_i^j}F_i(x) > 0$.
  \item $\exists M > 0$ such that if $x_i^j > M$ for all species in
    the $ith$ complex, then $F_i(x) > I$.
  \end{enumerate}
  \label{msc_assump}
\end{assumption}

As in the SSC case, if we want to add a random perturbation to the
input flux of the system, we must only consider perturbations that
will never drive concentrations into the negative portion of the real
line.  We handle this issue in a similar manner as in the SSC case:

If the perturbation is white in time, we multiply the perturbation by
a function which will go to zero if the concentration of one of the
species in the first complex goes to zero.  Therefore, let
$\theta_{\delta}(x_1^{1},\dots,x_1^{m_1}): \R^{m_1} \rightarrow
\R_{\ge 0}$ satisfy the following three properties for some small
$\delta >0$.
\begin{enumerate}
\item $\theta_{\delta}(x) = 1$ when each $x_1^j > \delta$.
\item $\theta_{\delta}(x) = 0$ if $x_1^j = 0$ for any $1 \le j \le m_1$.
\item $\theta_{\delta}$ is $C^{\infty}$ and is monotone increasing in
  each of the variables $x_1^{1},\dots,x_1^{m_1}$.
\end{enumerate}
If we add a white noise perturbation multiplied by
$\theta_{\delta}(x)$ to the input of the system, then the dynamics are
now governed by the stochastic differential equation
\begin{equation}
  dx(t) = f(x(t))dt + \sigma \theta_{\delta}(x)dB(t) u,
\label{mscsdes}
\end{equation}
where $u = [v_{11},v_{12},\dots, v_{1m_1},0,\dots,0]^T$, $\sigma \in
\R_{>0}$, $B(t)$ is standard one dimensional Brownian motion, and
$f_1^i(x) = v_{1i}(I - F_1(x))$ for each $1 \le i \le m_1$, $f_2^i =
v_{2i}(F_{1}(x) - F_{2}(x))$, for each $1 \le i \le m_2$, $\dots$,
$f_n^i = v_{ni}(F_{n-1}(x) - F_{n}(x))$ for each $1 \le i \le m_n$.

If the perturbation is a mean zero, finite variance, stationary
process, $\xi(t,\omega)$, that for each $\omega$ is continuous for
almost all $t$ and that satisfies the conditions of Appendix
\ref{assumptionXi}, then in order to keep the concentrations
non-negative, we again assume that $\xi(t,\omega) \ge -I$ for all $t$
and $\omega$.  In this case the dynamics of the system are governed by
the differential equation
\begin{equation}
  \dot x(t) = f(x(t)) + \xi(t)u, 
  \label{mscstationary}
\end{equation}
where $f$ and $u$ are as above.

The following four theorems are analogous to those in the SSC case and
their proofs can be found in Appendix \ref{msc_proof}.

\begin{theorem}[Decreasing variance down a nonlinear MSC chain]
  Let $x^*(t)$ be a stationary solution for the dynamics given by
  either equation \eqref{mscsdes} or \eqref{mscstationary}.  Then for
  all $i \ge 1$ and $t \ge 0$
  \begin{equation*}
    \var \left(F_{i}(x^*(t)) \right) > \var \left(F_{i+1}(x^*(t)) \right).
  \end{equation*}
  \label{thm:msc}
\end{theorem}

\begin{theorem}[Ergodicity of the MSC chain with white noise]
  Equation \eqref{mscsdes} possesses a unique invariant measure,
  $\mu$, on $\mathbb R^n$.  Furthermore, the distribution of any
  solution to equation \eqref{sscsdes} converges to $\mu$ as $t \to
  \infty$.
  \label{msc_measurelemma}
\end{theorem}

\begin{theorem} Let $\xi$ be as in Appendix \ref{assumptionXi}. Under
  the additional assumption that the rate functions $F_i$ are
  unbounded as $x \to \infty$, equation \eqref{mscstationary}
  possesses a unique stationary solution, $x^*(t)$. Furthermore, any
  solution $x(t)$ to equation \eqref{mscstationary} converges to
  $x^*(t)$ as $t \to \infty$.
  \label{cor:msc_arb_lineargrowth}
\end{theorem}

\begin{theorem} Let $\xi$ be as in Appendix \ref{assumptionXi}. Define
  $K = \displaystyle{\min_{i}\{\lim_{x \to \infty}F_i(x) - I\}}$.
  Under the additional assumption that $-I \le \xi(t) \le M < K$, for
  all $t$ and some $M < K$, equation \eqref{mscstationary} possesses a
  unique stationary solution, $x^*(t)$. Furthermore, any solution
  $x(t)$ to equation \eqref{mscstationary} converges to $x^*(t)$ as $t
  \to \infty$.
  \label{cor:msc_arb_bdd}
\end{theorem}

\begin{example}[Sum of Two Species with Mass Action Kinetics]
  Consider the following MSC chain with mass action kinetics
  \begin{equation*}
    \begin{array}{cccccccc}
      I(t) & & F_1 & & F_2 & & F_3 & \\
      \longrightarrow & Y & \longrightarrow & X_1 + X_2 &
      \longrightarrow & X_3 + X_4 & \longrightarrow, 
    \end{array}
  \end{equation*}
  where $I(t) = 10 + 2\theta_{\delta}(y) dB_t$ (with $\delta = .001$),
  $F_1(y) = y$, $F_2(x_1,x_2) = x_1x_2$, $F_3(x_3,x_4) = x_3x_4$.
  Using Matlab to perform a Monte Carlo simulation we computed the
  means, variances, and coefficient of variation of the fluxes to be
  the following:
  \begin{center}
    \begin{tabular}{|c|c|c|c|}
      \hline
      & $F_1(y)$ & $F_2(x_1,x_2)$ & $F_3(x_3,x_4)$  \\
      \hline
      mean & 10 & 10 & 10 \\
      \hline
      variance & 2 & 1.73 & 1.62  \\
      \hline
      CV & 0.14 & .131 & .127\\
      \hline
    \end{tabular}
  \end{center}
  We note that the variances and coefficients of variation of the
  fluxes decrease down the chain, as guaranteed by Theorem
  \ref{thm:msc}.
\label{MAexample}
\end{example}

\begin{example}[Sum of Two Species with Michaelis-Menten Kinetics]
    Consider the following MSC chain with Michaelis-Menten Kinetics
  \begin{equation*}
    \begin{array}{cccccccc}
      I(t) & & F_1 & & F_2 & & F_3 & \\
      \longrightarrow & Y & \longrightarrow & X_1 + X_2 &
      \longrightarrow & X_3 + X_4 & \longrightarrow, 
    \end{array}
  \end{equation*}
  where $I(t)$ and $F_1$ are as in Example \ref{MAexample}, and
  $F_2(x_1,x_2) = 14x_1x_2/[(1 + x_1)(1+x_2)]$, $F_3(x_3,x_4) =
  14x_3x_4/[(1 + x_3)(1+x_4)]$.  Using Matlab to perform a Monte Carlo
  simulation we computed the means, variances, and coefficient of
  variation of the fluxes to be the following:
  \begin{center}
    \begin{tabular}{|c|c|c|c|}
      \hline
      & $F_1(y)$ & $F_2(x_1,x_2)$ & $F_3(x_3,x_4)$  \\ 
      \hline
      mean & 10 & 10 & 10  \\
      \hline
      variance & 2 & .72 & .49  \\
      \hline
      CV & .14 & .085 & .07 \\
      \hline
    \end{tabular}
  \end{center}
  As guaranteed by Theorem \ref{thm:msc} the variances and
  coefficients of variation of the fluxes decrease down the chain.
  \label{biMMexample}
\end{example}

\begin{example}[Species can not be in more than one complex]
  Consider the following MSC chain subjected to white noise
  perturbations for which the species $X_1$ appears in two complexes
  (and so this system is not covered by Theorem \ref{thm:msc})
  \begin{equation*}
    \begin{array}{ccccccc}
      I(t) & & F_1 & & F_2 &  & F_3\\
      \longrightarrow & X_1 + X_2 & \longrightarrow & X_3 &
      \longrightarrow & X_1 + X_4 & \longrightarrow ,
    \end{array}
  \end{equation*}
  where $I(t) = 10 + \theta_{\delta}(x_1,x_2)dB(t)$ (with $\delta =
  .001$), $F_1(x_1,x_2) = 2x_1x_2$, $F_2(x_3) = x_3$, and
  $F_3(x_1,x_4) = 5x_1x_4$.  We performed a Monte Carlo simulation
  using Matlab to compute:
  \begin{center}
    \begin{tabular}{|c|c|c|c|}
      \hline
      & $F_1(x_1,x_2)$ & $F_2(x_3)$ & $F_3(x_1,x_4)$ \\
      \hline
      mean & 10 & 10 & 10  \\
      \hline
      variance & 4.16 & .45 & 1.71\\
      \hline
      CV & .204 & .067 & .131 \\
      \hline
    \end{tabular}
\end{center}
Note that $\var (F_3(x)) > \var (F_2(x))$ and $CV(F_3(x)) > CV
(F_2(x))$.  Therefore, the assumption in Theorem \ref{thm:msc} that
each species is in precisely one complex is necessary.
\label{example2}
\end{example}

\section{Discussion}

We have proven under a variety of different contexts that if the input
to a non-reversible biochemical reaction chain is perturbed by a
random process, then the variances and coefficients of variation of
the fluxes will decrease as one moves down the chain.  The assumptions
made on the different choices of perturbations and on the properties
of the rate functions were varied, and explicitly spelled out.
Further, much care was taken to state precisely what is meant by
``fluctuations decrease down reaction chains.''  Due to this
(necessary) mathematical detail, however, it is easy for the
over-riding point of the paper to be lost: considering the
out-of-equilibrium dynamics of a biochemical system can be an
important tool for understanding the dynamical properties of that
system.

A comparison of the results of this paper to metabolic control
analysis (MCA) (\cite{KacserBurns73}\cite{HeinRap74}) sheds light on
the importance of considering out-of-equilibrium dynamics.  The
control coefficient for the flux out of a reaction chain, $F$, in
terms of the input, $I$, is
\begin{equation*}
C^F_I = \frac{\partial F}{\partial I} \cdot \frac{I}{F},
\end{equation*}
where the values are computed at equilibrium.  However, at
equilibrium, $F = I$.  Therefore, independent of the choice of
reaction kinetics or the length of the chain, $C^F_I = 1$.  This
implies that changes in the output of a chain correspond directly with
changes to the input.  However, by studying the out-of-equilibrium
dynamics, we have shown in this paper that the fluctuations in a
reaction chain will actually {\it decrease} as one moves down a
reaction chain and changes to the output of a chain do not correspond
directly with changes to the input.  The differing results are
biologically significant since it is tempting to speculate that this
decrease in fluctuations (and, hence, increase in stability) is one
reason long reaction chains may be evolutionarily advantageous in
cellular systems.

While all of the technical details of the proofs have been relegated
to the appendices, we would like to point out that to prove the main
results of this paper (with the exception of Theorem \ref{thm:avg})
two things must be shown: 1) the existence of a unique solution whose
statistics are stationary for the dynamics and 2) variances of a
stationary solution decrease down reaction chains.  Typically, a
stationary solution can be proved to exist so long as the
perturbations to the system do not drive any solutions to infinity.
Stability properties of the non-perturbed system can then be used to
show uniqueness of the stationary solution.  The fact that the
variances decrease down reaction chains follows from standard
inequalities and the use of Lyapunov type functions.  Intuitively,
however, the variances decrease down reaction chains because the
dynamics are always forcing the output flux from a complex towards the
input flux.  That is, the dynamics are constantly moving the system
towards a shifting equilibrium.  There is, however, a natural time
delay in its ability to do so.  Therefore the output will always be
lagging behind the input, which leads to the decrease in variance.

There is still much work to be done in studying biochemical reaction
systems subjected to external perturbations.  A natural extension of
this work and that of \cite{AndMattReed06} would be to attempt to
analyze reaction systems with more complicated geometries and more
complicated kinetics (like product inhibition).  The main technical
issues encountered in such a study would be: 1) the extremely weak
stability of many such systems (\cite{Feinberg87}\cite{FeinbergLec79})
would make proving the existence of a stationary solution difficult and
2) it will be difficult to isolate the variances of particular fluxes
or specie concentrations within a complicated system.  While both of
these problems are formidable in a theoretical study such as in this
paper, they become trivial in an {\em in silico} study
(\cite{nijhout06}).

\vspace{.25in}
\begin{center}
  {\bf \large Acknowledgments}
\end{center}

This paper is part of a larger research program between the authors,
Fred Nijhout, and Michael C. Reed, and we thank them for their input
and encouragement. In particular, we thank Michael C. Reed for being
the catalyst of this line of research. We also acknowledge useful
discussions with Martin Hairer. DFA was supported by the grants
NIH R01-CA105437, NSF DMS-0109872 and NSF DMS-0553687.  JCM was
partially supported by an NSF CAREER award (DMS0449910) and an Alfred
P.  Sloan foundation fellowship.

\appendix

\section{Precise definition of  the noise $\xi$ and the proofs}
\label{appA}

\subsection{Assumptions on the noise $\xi(t)$ needed for existence}

\label{assumptionXi}
In addition to the standing assumptions that $\xi(t)$ is stationary
with mean zero and finite variance, to prove the existence of a
stationary solution for \eqref{sscstation}, we need to assume that the
distribution of $\xi(t)$ is determined entirely by the past of $\xi$
on any interval of time $(-\infty,s]$ with $s \leq t$. Intuitively, we
mean that given the value of $\xi(s)$ for $s\in (-\infty,t]$ the
distribution on $[t,T]$ is uniquely determined for any $T>t$. This
must be done in a way such that if one first adds a segment $[t,T-r]$
and then $[T-r,T],$ the resulting distribution on $[t,T]$ is the same
as if one had added the segment $[t,T]$ in one step. For a discussion
of some of the issues involved if one does not make such an
assumption, see \cite{Hairer05} and subsequent works by the author.

Let $C_{ae}((-\infty,0],\R)$ denote the space of almost every where
continuous functions, $f$, endowed with the norm $\sup |f(s)|
e^{-\alpha|s|}$ for some $\alpha >0$.  Let $\mathcal{P}_t$ be a Markov
semigroup on $C_{ae}((-\infty,0],\R)$ which is Feller, has an
invariant measure $\mathcal{M}$, and such that for
$\mathcal{M}$-almost every $\gamma_0 \in C_{ae}((-\infty,0],\R)$,
$\mathcal{P}_t(\gamma_0, \;\cdot\;)$ is concentrated on elements
$\gamma_t \in C_{ae}((-\infty,0],\R)$ with $\gamma_t(s) =
\gamma_0(s+t)$ for $s\leq -t$. If $\gamma_t$, $t\geq 0$, is a
realization of the Markov chain generated with $\mathcal{P}_t$ with
$\gamma_0$ distributed as $\mathcal{M}$, we define
$\xi(s)=\gamma_t(s-t)$ for $s \leq t$. This is well defined since our
assumptions on $\mathcal{P}_t$ make $\gamma_s(r-s)=\gamma_t(r-t)$ for
$r \leq \min(s,t)$.

The dynamics of such a $\mathcal{P}_t$ can be understood as follows.
Given an initial history from $-\infty$ to $0$, one adds on a segment
of length $t$, resulting in a trajectory from $-\infty$ to $t$.  After
shifting this trajectory back by $-t$, one again obtains a trajectory
from $-\infty$ to $0$.  The distribution of this new trajectory from
$-\infty$ to $0$ is given by $\mathcal{P}_t$.  The conditions above
simply insure that the trajectory from $-\infty$ to $-t$ coincide with
the initial trajectory from $-\infty$ to zero.  If $\xi(t)$ is a
Markov process, then it can be constructed as above and hence is an
example of the type of noise we allow.

\subsection{Proof of principal result on variances}
\label{appA_principal}

\begin{proof}[Proof of Theorem \ref{nonlineardecrease}]
  We consider the dynamics given by equation \eqref{sscsdes}.  The
  proof when the dynamics is given by equation \eqref{sscstation} is
  identical.

  Defining $\xi_1(t) = F_1(x_1^*(t)) - I$, the equations governing
  $x_1^*$ and $x_2^*$ are
  \begin{align}
    d x_1^* &= (I - F_1(x_1^*))dt + \sigma \theta_{\delta}(x_1^*) dB(t)
    \label{equ1}\\
    \dot x_2^* &= F_1(x_1^*) - F_2(x_2^*) \ \dot = \ I - F_2(x_2^*) + \xi_1.
    \label{equ2}
  \end{align}
  We claim that for any $t$, $\E F_1(x_1^*(t)) = \E F_2(x_2^*(t)) =
  I$.  Integrating \eqref{equ1}, taking expected values, using that
  the distribution of $x^*(t)$ is stationary and noting that $\sigma
  \int_0^t \theta_{\delta}(x_1^*(s)) dB_s$ is an $L^2 - $martingale
  gives
  \begin{equation*}
    \E x_1^*(t) = \E x_1^*(0)+ \left(I - \E F_1(x_1^*(t))\right)t.
  \end{equation*}
  By the stationarity of the system, $\E x_1^*(t) = \E x_1^*(0)$, so
  $\E \left(F_1(x_1^*(t)) - I \right) = 0$ as claimed.  A similar
  argument which uses $\E F_1(x_1^*(t)) = I$ (and, hence, $\E \xi_1(t)
  = 0$) shows that $\E F_2(x_2^*(t)) = I$.  Therefore, in order to
  show that $\var(F_2(x_2^*(t))) < \var(F_1(x_1^*(t)))$, we need $\E
  \left( F_2(x_2^*(t)) - I \right)^2 < \E \xi_1(t)^2$.

  Let $G_1(x) = 2\int_0^x (F_2(y) - I)dy$.  Then,
  \begin{align}
    \begin{split}
      \frac{d}{dt}G_1(x_2^*(t)) &= G_1'(x_2^*(t)) \dot x_2^*(t)\\
      &= 2(F_2(x_2^*(t)) - I)(I - F_2(x_2^*(t)) + \xi_1(t))\\
      &= -2(F_2(x_2^*(t)) - I)^2 + 2(F_2(x_2^*(t)) - I)\xi_1(t).
    \end{split}
    \label{Gfunc}
  \end{align}
  Pick $\overline t > 0$ arbitrarily.  Integrating \eqref{Gfunc} up to time
  $\overline t$ and taking expected values gives
  \begin{align}
    \begin{split}
      \E G_1(x_2(\overline t)) - \E G_1(x_2(0)) = &-2\int_0^{\overline
        t}\E (F_2(x_2(s)) -
      I)^2 ds\\
      & + 2\int_0^{\overline t}\E \left[(F_2(x_2(s)) -
        I)\xi_1(s)\right] ds
    \end{split}
    \label{equality}
  \end{align}
  Using that $x_2^*(t)$ is stationary, differentiation of
  \eqref{equality} together with the inequality $2ab \le a^2 + b^2$
  gives
  \begin{align}
    \begin{split}
      0 &= -2 \E (F_2(x_2^*(\overline t)) - I)^2 + 2 \E
      \left[(F_2(x_2^*(\overline t)) - I)\xi_1(\overline t)\right]\\
      & \le -\E (F_2(x_2^*(\overline t)) - I)^2 + \E \xi_1(\overline
      t)^2.
      \label{inequality}
    \end{split}
  \end{align}
  We claim, however, that the inequality in \eqref{inequality} is
  strict.  To see why, we suppose, in order to find a contradiction,
  that $2\E \left[(F_2(x_2^*(\overline t)) - I)\xi_1(\overline
    t)\right] = \E (F_2(x_2^*(\overline t)) - I)^2 + \E
  \xi_1(\overline t)^2.$ Then $F_2(x_2^*(\overline t)) - I =
  \xi_1(\overline t)$ with probability one.  However, this implies
  $F_2(x_2^*(\overline t)) = F_1(x_1^*(\overline t))$ with probability
  one.  Because $\overline t$ was arbitrary, we conclude that with
  probability one $F_2(x_2^*(t)) = F_1(x_1^*(t))$ for all $t$ in some
  countably dense subset of $\R$.  However, by the continuity of the
  functions involved, this implies that with probability one
  $F_2(x_2^*(t)) = F_1(x_1^*(t))$ for all $t \in \R$.  Thus, $\dot
  x_2^*(t) = 0$ for all time and $x_2^*(t)$ is a constant.  But,
  $F_2(x_2^*(t)) \equiv F_1(x_1^*(t))$ and so $x_1^*(t)$ is also a
  constant.  However, for any $t > 0$, $P\{x_1^*(s) = const: s \in
  [0,t]\} = 0 \ne 1$.  Thus, the inequality in \eqref{inequality} is
  strict, which was the desired result of the Theorem.  Therefore, the
  result is shown for the first step in the chain.  To complete the
  proof, one simply repeats the argument down the chain.
\end{proof}

\subsection{Existence and uniqueness of stationary solutions and
  invariant measures}
\label{appA_measures}

The proofs of Theorems \ref{measurelemma}, \ref{cor:arb_lineargrowth},
and \ref{cor:arb_bdd} have the same overall structure. We use the
assumptions over the dynamics to obtain a uniform in time bound on
some statistic of the concentration vector which can be used to prove
that a sequence of time averages is tight. By extracting a convergent
sub-sequence we can prove the existence of at least one invariant
measure for the white in time setting. For the stationary forcing, we
must work on the space to trajectories stretching back to negative
infinity and prove the existence of a stationary measure on that
space.  We then prove that the invariant measure or stationary
solution is unique and that the statistics of any solution converge to
it under the dynamics of the system.

To prove the needed tightness for Theorem \ref{measurelemma}, we make
use of the following Lyapunov function:
\begin{equation}
  V(x) = \sum_{i=1}^n \frac{V_i}{2}\left[\sum_{j=1}^i \left (x_j -
      \bar x_j \right) \right]^2,
  \label{lyapunov}
\end{equation}
where the $V_i$'s are positive numbers yet to be determined and the
$\bar x_j$ are defined as the solution to $F_j(\bar x_j)=I$ (that
is, they are the equilibrium values of the unperturbed problem).  As
an example, for a chain with $n = 2$ we have
\begin{equation*}
  V(x) = \frac{V_1}{2}(x_1 - \bar x_1)^2 + \frac{V_2}{2}\left[\left(x_1
      - \bar x_1 \right) + (x_2 - \bar x_2) \right]^2.
\end{equation*}
We begin by proving a fact that, while technical, is the crux of the
proof of Theorem \ref{measurelemma}.

\begin{lemma}
  Let $\mathcal {A}$ be the generator of the SDEs \eqref{sscsdes}.
  Then there are positive numbers $V_1,V_2,\dots,V_n$ and positive
  numbers $c,k$ such that if $V(x)$ is defined by \eqref{lyapunov}
  then $\mathcal{A} V(x) \le c - k|x|$.
  \label{fact}
\end{lemma}

\begin{proof}
  For all $k \le n$:
  \begin{equation}
    \frac{\partial V}{\partial x_k} = \frac{\partial}{\partial x_k}\sum_{i=1}^n
    \frac{V_i}{2}\left[\sum_{j=1}^i \left (x_j - \bar x_j \right)
    \right]^2 = \sum_{i=k}^n V_i\left[\sum_{j=1}^i \left (x_j - \bar x_j
      \right) \right].
    \label{dV}
  \end{equation}
  Let $F_0 = I$.  Using equation \eqref{dV}, it can be shown that
  \begin{equation*}
    \sum_{k=1}^n \frac{\partial V}{\partial x_k}(F_{k-1}-F_k) =
    \sum_{j=1}^n (x_j - \bar x_j) \left(\sum_{i=j}^n V_i (I-F_i)\right).
  \end{equation*}
  Therefore,
  \begin{align*}
    \mathcal{A}V(x) &= \frac{1}{2}\sigma^2 \theta_{\delta}(x_1)^2
    \frac{\partial^2}{\partial x_1^2}V(x) +
    \sum_{k=1}^n \frac{\partial V}{\partial x_k}(F_{k-1}-F_k)\\
    &= \frac{1}{2}\left(\sigma^2 \theta_{\delta}(x_1)^2 \sum_{i=1}^n V_i
    \right) + \sum_{j=1}^n (x_j - \bar x_j) \left(\sum_{i=j}^n V_i
      (I-F_i)\right)\\
    &\ \dot = \ \frac{1}{2}\left(\sigma^2 \theta_{\delta}(x_1)^2 \sum_{i=1}^n V_i
    \right) + \sum_{j=1}^n s_j(x),
  \end{align*}
  where the last equality is a definition.  We now choose the
  $V_j$'s recursively.  Let $V_n = 1$.  Because
  $\displaystyle{\lim_{x \to \infty} F_n(x) > I}$, $s_n$ is bounded
  by:
  \begin{equation*}
    s_n(x) = (x_n -\bar x_n)(I - F_n(x_n)) < c_n - k_n x_n,
  \end{equation*}
  where $c_n$ and $k_n$ are some positive constants.  Then $s_{n-1}$
  is given by
  \begin{align*}
    s_{n-1}(x) &= (x_{n-1} - \bar x_{n-1}) \left( V_{n-1}(I -
      F_{n-1}(x_{n-1})) + (I - F_n(x_n))\right).
  \end{align*}
  $F_n(x_n) \ge 0$, so if $x_{n-1} \ge \bar x_{n-1}$, then
  \begin{align*}
    s_{n-1}(x) &\le (x_{n-1} - \bar x_{n-1}) \left( V_{n-1}(I -
      F_{n-1}(x_{n-1})) + I\right).
  \end{align*}
  We may therefore choose $V_{n-1}$ to be large enough so that there
  are positive constants $c_{n-1}$ and $k_{n-1}$ such that
  \begin{equation*}
    s_{n-1}(x) < c_{n-1} - k_{n-1}x_{n-1}.
  \end{equation*}
  Continuing up the chain, we consider $s_j$ for $j < n$.  When $x_j >
  \bar x_j$ we have
  \begin{align*}
    s_{j}(x) &= (x_j - \bar x_j) \left( V_j(I - F_j(x_j)) + \sum_{i=j+1}^n
      V_i (I-F_i)\right)\\
    & \le (x_j - \bar x_j) \left( V_j(I - F_j(x_j)) + I \sum_{i=j+1}^n
      V_i \right).
  \end{align*}
  Since $V_{j+1} \dots, V_n$ have already been defined, we may choose
  $V_j$ so large that there are positive constants $c_j$ and $k_j$
  such that
  \begin{equation*}
    s_{j} < c_{j} - k_{j}x_{j}.
  \end{equation*}
  Setting
  \begin{equation*}
    c = \frac{1}{2}\sigma^2 \sum_{i=1}^n V_i  + \sum_{i=1}^n c_i,
  \end{equation*}
  we now have that for some $k$
  \begin{equation*}
    \mathcal{A}V(x) \le c - \sum_{i=1}^n k_i x_i \le c - k|x|,
  \end{equation*}
  which was the desired result.
\end{proof}

\begin{proof}[Proof of Theorem \ref{measurelemma}]
  The proof has two parts.  First, we will use Lemma \ref{fact} and
  Prohorov's theorem (\cite{Bill91Conv}, pg. 59) to show that there
  exists a measure which is invariant to the stochastic flow generated
  by equation \eqref{sscsdes}.  We will then prove that this invariant
  measure is unique and that all distributions converge to it under
  the flow of the SDE \eqref{sscsdes}.

  \lskip

  \noindent {\it Part I}

  Let $V(x)$ be defined by \eqref{lyapunov} where
  $V_1,V_2,\dots,V_n$ are given by Lemma \ref{fact}.  Then, if $k,c >
  0$ are the constants given in the conclusion of Lemma \ref{fact},
  \begin{equation*}
    dV(x) = \mathcal{A}V(x)dt + dM(t) \le \left(c - k|x|\right)dt + dM(t),
  \end{equation*}
  where $M(t)$ is some $L^2$ - martingale.  Integrating gives
  \begin{align*}
    V(x(t)) \le V(x(0)) + ct - k\int_0^t |x(s)|ds + M(t) - M(0),
  \end{align*}
  where $x(0)$ is some fixed value.  Rearranging terms, taking
  expected values and using the fact that $\E M(t) = \E M(0)$ then
  yields,
  \begin{equation*}
    \frac{1}{t}\int_0^t \E |x(s)|ds \le \frac{c}{k} + \frac{V(x(0))}{kt}.
  \end{equation*}
  Thus for any $R>0$ Chebychev's inequality gives
  \begin{equation}
    \frac{1}{t}\int_0^t P\{|x(s)|>R\} ds \le \frac{c}{k}
    \frac{1}{R} + \frac{V(x(0))}{kt}\frac{1}{R},
    \label{kb1}
  \end{equation}
  where, again, the initial condition $x(0)$ is fixed.  The right side
  of equation \eqref{kb1} converges to zero uniformly in $t \ge 1$ as
  $R \to \infty$.  Therefore, the sequence of measures on $\mathbb R$
  defined by
  \begin{equation*}
    \nu_n(A) \ \dot= \ \frac{1}{t_n}\int_0^{t_n} P\{x(s) \in A\} ds,
  \end{equation*}
  where $t_n \to \infty$ as $n\to \infty$ and $A \subset \mathbb R^n$,
  is tight (\cite{Bill91Conv}, pg. 59).  By Prohorov's theorem,
  $\nu_n$ is relatively compact and so there exists a subsequence
  $\nu_{n_k}$ and a measure $\mu$, such that $\nu_{n_k} \to \mu$,
  where the convergence is weak convergence.  Thus, for all $A \subset
  \mathbb R$
  \begin{equation*}
    \mu(A) = \lim_{k \to \infty} \frac{1}{t_{n_k}}\int_0^{t_{n_k}}
    P\{x(s) \in A \} ds.
  \end{equation*}
  For $A \subset \mathbb R^n$, let $\phi_T(A) = \{x(T): x(0) \in A\}$.
  To show that $\mu$ is invariant to the flow of equation
  \eqref{sscsdes} we need to demonstrate that for all $T>0$ and $A
  \subset \mathbb R$, $\mu(\phi^{-1}_T)(A) = \mu(A)$, where
  $\mu(\phi^{-1}_T)(A) \; \dot = \; \mu(x : \phi_T(x) \in A)$.  Note
  that, by definition, $\mu(x : \phi_T(x) \in A) = \lim_{k \to
    \infty} \frac{1}{t_{n_k}}\int_0^{t_{n_k}} P\{x(s+T) \in A \} ds$.
  Using a change of variable, we then make the following computation for
  any $T>0$ and $A \subset \mathbb R$:
  \begin{align*}
    \mu(\phi^{-1}_T)(A) &= \lim_{k \to \infty}
    \frac{1}{t_{n_k}}\int_0^{t_{n_k}} P\{x(s+T) \in A \} ds\\
    &= \lim_{k \to \infty} \frac{1}{t_{n_k}}\int_0^{t_{n_k}} P\{x(s)
    \in A \} ds + \lim_{k \to \infty}
    \frac{1}{t_{n_k}}\int_{t_{n_k}}^{t_{n_k}+T} P\{x(s) \in A \} ds\\
    & \ \ \ \ \ - \lim_{k \to \infty} \frac{1}{t_{n_k}}\int_{0}^{T}
    P\{x(s) \in A \} ds\\
    &= \mu(A) + \lim_{k \to \infty} \frac{1}{t_{n_k}} \left[
      \int_{t_{n_k}}^{t_{n_k}+T} P\{x(s) \in A \} ds - \int_{0}^{T}
      P\{x(s) \in A \} ds \right].
  \end{align*}
  However,
  \begin{align*}
    \left|\lim_{k \to \infty} \frac{1}{t_{n_k}} \left[
        \int_{t_{n_k}}^{t_{n_k}+T} P\{x(s) \in A \} ds - \int_{0}^{T}
        P\{x(s) \in A \} ds \right] \right| & \le \lim_{k \to
      \infty}\frac{2T}{t_{n_k}} = 0,
  \end{align*}
  and so $\mu(\phi^{-1}_T)(A) = \mu(A)$.  Thus, $\mu$ is invariant
  under the stochastic flow generated by \eqref{sscsdes}.

  \lskip

  \noindent {\it Part II}

  The proof that the invariant measure is unique is not completely
  straightforward.  The noise enters only one species, hence the
  diffusion is not uniformly elliptic (so arguments such as in
  \cite{MackeyLasota94} do not suffice).  The proof we now sketch
  follows a rather standard line of argument. We refer the reader to
  \cite{mattStuartH} \cite{BelletThomas2002EC} \cite{MeynTweedie93MCS}
  for the missing details.  The proof has three elements. First, one
  shows that the generator of the diffusion satisfies H\"ormander's
  ``sum of squares'' theorem and hence is hypoelliptic. This ensures
  that the Markov transition density $p_t(x,y)$ is smooth in $x$ and
  $y$ and hence is a Strong Feller process.  This gives the local
  smoothing needed to ensure that the invariant measure found above is
  unique. The structure of \eqref{sscsdes} and the fact that the
  $F_i'$ do not vanish ensures that the span of the needed Lie
  brackets is of full dimension. Hence H\"ormander's theorem holds.
  
  Secondly, we need to provide the global information which ensures
  open set irreducibility (the fact that processes starting from
  different initial points have nonzero probability of entering a
  small neighborhood of each other). The Lyapunov function given by
  \eqref{lyapunov} shows that the processes return to a bounded ball
  $\mathcal{B}$ about the origin eventually.  Since there is a
  globally attracting fix point, if the noise is small for long enough
  all of the points of $\mathcal{B}$ will enter an arbitrarily small
  neighborhood of the fixed point.
  
  Finally, with the above facts in hand, the uniqueness and
  convergence result follows from standard arguments (see
  \cite{b:MeynTweedie93c} \cite{BelletThomas2002EC} \cite{mattStuartH}
  \cite{b:EMattingly00}).

\end{proof}

\begin{proof}[Proof of Theorem \ref{cor:arb_lineargrowth}]

  As in the proof of Theorem \ref{measurelemma}, the proof is split
  into two parts.  In the first we prove the existence of a stationary
  solution $x^*(t)$ for the dynamics \eqref{sscstation}.  In the
  second we show that the if $x(t)$ and $y(t)$ are solutions driven by
  the same noise, then $y(t) \to x(t)$ pathwise.  Hence, we conclude
  there can only be one stationary solution since any two would
  converge to each other over time.

  \noindent {\it Part I}\\
  Unlike the previous example, the process $x(t)$ alone is not a
  Markov process.  However, if we include the entire history of $\xi$
  then the system does become Markovian. More precisely, from the
  assumptions in section \ref{assumptionXi} we know that
  $\xi(s)=\gamma_t(s-t)$ for $s \leq t$ where $\gamma_t$ is a Feller
  Markov process on $C_{ae}((-\infty,0],\R)$ with semi-group
  $\mathcal{P}_t$ and with $\gamma_0$ distributed as the invariant
  measure $\mathcal{M}$.  Then the pair $(x(t),\gamma_t)$ is a Markov
  process on the expanded state space $\R^n\times
  C_{ae}((-\infty,0],\R)$. Let $ \mathcal{\hat P}_t$ denote the Markov
  transition semi-group of this system and $\pi_x$ and $\pi_\gamma$
  the projection onto the $x$ and $\gamma$ coordinate respectively.
  Since we start $\gamma_t$ from an invariant measure for its
  dynamics, we know that the statistics of $\gamma_t$ are constant in
  time equal to $\mathcal{M}$ for all $t$ and hence is tight.

  Let $x^{(0)}$ be an arbitrary initial condition for $x(t)$ Defining
  the measure
  \begin{align*}
    Q_t(\;\cdot\;)= \frac1t \int_0^t \int \mathcal{\hat P}_s(x^{(0)},
    \gamma,\;\cdot\;) \mathcal{M}(d\gamma) ds
  \end{align*}
  we need only show that $Q_t \pi_x^{-1}$ is tight to conclude that
  $Q_t$ is tight since $Q_t \pi_\gamma^{-1}=\mathcal{M}$ is
  independent for $t$. We will do this coordinate by coordinate.
  Consider the equation governing $x_1(t)$:
\begin{equation}
    \dot x_1(t) = I - F_1(x_1(t)) + \xi(t).
    \label{eq:x_1}
\end{equation}
Integrating \eqref{eq:x_1} gives
  \begin{align}
    x_1(t) &= x(0) + It - \int_0^t F_1(x_1(s))ds + \int_0^t \xi(s)
    ds  \label{intbound} \\
    &\le x(0) + It - \int_0^t F_1(x_1(s)) 1_{\{|x(s)| > R\}}ds +
    \int_0^t \xi(s) ds \notag\\
    &\le x(0) + It - F_1(R) \int_0^t 1_{\{|x_1(s)| > R\}}ds + \int_0^t
    \xi(s) ds. \notag
  \end{align}
  Taking expected values and rearranging terms gives
  \begin{equation}
 \frac{1}{t}\int_0^t P\{|x_1(s)| > R\} ds \le \frac{I}{F_1(R)} +
    \frac{\E x_1(0)}{F_1(R)}.
    \label{firstbd}
  \end{equation}
  Note that rearranging equation \eqref{intbound} and taking expected
  values gives us the additional bound
  \begin{equation}
  \frac{1}{t}\int_0^t \E F_1(x_1(s))ds \le I + \frac{\E x_1(0)}{t}.
  \label{intbound2}
  \end{equation}

  Continuing down the chain we consider $x_2(t)$:
  \begin{equation*}
  \dot x_2(t) = F_1(x_1(t)) - F_2(x_2(t)).
  \end{equation*}
  Integrating gives
  \begin{align}
  x_2(t) &= x_2(0) + \int_0^t F_1(x_1(s)) ds - \int_0^t
  F_2(x_2(s))ds \label{another}\\
  &\le x_2(0) + \int_0^t F_1(x_1(s)) ds - F_2(R)\int_0^t
  1_{\{|x_2(s)| > R\}}ds \notag
  \end{align}
  Rearranging terms as before, taking expected values and using
  equation \eqref{intbound2} gives
  \begin{equation*}
  \frac{1}{t}\int_0^t P\{|x_2(s)| > R\} ds \le \frac{I}{F_2(R)} +
    \frac{\E x_1(0)}{tF_2(R)} + \frac{\E x_2(0)}{tF_2(R)}
  \end{equation*}
  Further, rearranging equation \eqref{another} and using equation
  \eqref{intbound2} gives
  \begin{equation*}
    \frac{1}{t}\int_0^t \E F_2(x_2(s))ds \le I + \frac{\E x_1(0)}{t} +
    \frac{\E x_2(0)}{t}. 
  \end{equation*}
  We may continue down the chain in a similar manner and conclude
  that there are positive constants $c_1, c_2,\dots,c_n$ such that for all $t
  \ge 1$
  \begin{align}
    \label{totalbound}    
    Q_t\pi_x^{-1}(\{ y: \sup_i |y_i| > R \})&= \frac{1}{t}\int_0^t
    P\{\sup_i |x_i(s)| > R \}ds \\ \notag &< \sum_i
    \frac{1}{t}\int_0^t P\{|x_i(s)| > R \}ds
    < \sum_i \frac{c_i}{F_i(R)}.
      \end{align}
      Because each $F_i$ is monotone and unbounded the right side of
      inequality \eqref{totalbound} above converges to zero uniformly
      in $t$ as $R \to \infty$.  Therefore, just as in the proof of
      Theorem \ref{measurelemma}, we may invoke Prohorov's Theorem to
      guarantee the existence of a measure $\mu$ on $\R^n \times
      C_{ae}((-\infty,0], \R)$ that is invariant to the dynamics
      induced by $\mathcal{\hat P}_t$.  By Kolmogormov's extension
      theorem we can use this measure to define a measure on pairs of
      noise $\xi$ and solution trajectories $x$ starting at $-\infty$
      and continuing to $\infty$. The projection of this measure onto
      the solution coordinate produces a stationary solution for the
      $x(t)$ dynamics. One should really view this stationary solution
      $x^*$ along with its noise trajectory $\xi$ which was
      constructed along with it.

      \noindent {\it Part II}

  Let $x^*(t) = x^*(t,\xi)$ be the stationary solution and matching
  noise trajectory found above. Let $y(t)$ be the solution starting
  from an arbitrary initial condition $y(0)$ using the same noise
  $\xi(t)$.

  Consider $x_1^*(t)$ and $y_1(t)$.  If $x_1^*(0) = y_1(0)$, then $x_1^*(t)
  = y_1(t)$ for all time by uniqueness of solutions.  Suppose that
  $x_1^*(0) > y_1(0)$ (if $x_1^*(0) < y_1(0)$ there is a symmetric
  argument).  Then, $x_1^*(t) > y_1(t)$ for all time.  Differentiating
  $x_1^*(t) - y_1(t)$ gives
  \begin{align*}
    \ddt (x_1^*(t) - y_1(t)) &= -(F_1(x_1^*(t)) - F_1(y_1(t)))\\
    &= -(x_1^*(t) - y_1(t))\frac{F_1(x_1^*(t)) - F_1(y_1(t))}{x_1^*(t) -
    y_1(t)}.
  \end{align*}
  Defining
  \begin{equation}
  H_1(t) = \frac{F_1(x_1^*(t)) - F_1(y_1(t))}{x_1^*(t) - y_1(t)},
  \label{eq:H}
  \end{equation}
  we have that
  \begin{align*}
  \begin{split}
    x_1^*(t) - y_1(t) &= (x_1^*(0) - y_1(0))e^{-\int_0^t H_1(s)ds}.
  \end{split}
  \end{align*}
  By the uniform bound given in equation \eqref{totalbound}, we know
  that both $x_1^*(t)$ and $y_1(t)$ spend a positive fraction of time in
  a compact set on which $H_1(t) > d_1 > 0$ for some $d_1 > 0$ (since
  $H_ 1$ is an approximation to the derivative of $F_1$).  Thus,
  $x_1^*(t) - y_1(t) \to 0$, as $t \to \infty$.

  We next consider $x_2^*(t)$ and $y_2(t)$.  Suppose that $x_2^*(0) <
  y_2(0)$.  Let $\tau_2$ be the first time $x_2^*(t) = y_2(t)$.  Then,
  up until time $\tau_2$,
  \begin{align*}
  \ddt (y_2(t) - x_2^*(t)) = -(F_1(x_1^*) - F_1(y_1)) - (F_2(y_2(t)) -
  F_2(x_2^*(t))).
  \end{align*}
  But, $x_1^* \ge y_1$ so $(F_1(x_1^*) - F_1(y_1)) \ge 0$ and
  \begin{align*}
    \ddt (y_2(t) - x_2^*(t)) &\le - (F_2(y_2(t)) - F_2(x_2^*(t)))\\
    &= -(y_2(t)- x_2^*(t))\frac{F_2(y_2(t)) - F_2(x_2^*(t))}{y_2(t) -
      x_2^*(t)}.
  \end{align*}
  Defining $H_2(t)$ as we did $H_1(t)$ we may conclude from the above
  that up until time $\tau_2$
  \begin{align*}
  y_2(t) - x_2^*(t) &< (y_2(0) - x_2^*(0))e^{-\int_0^t H_2(r) dr}.
  \end{align*}
  Therefore, if $\tau_2 = \infty$, then, as in the previous case, we
  may use the above equation and the bound \eqref{totalbound} to
  conclude that $|x_2^*(t) - y_2(t)| \to 0$, which is the desired
  result.

  If $\tau_2$ is finite, then for all time after $\tau_2$, $x_2^*(t) \ge
  y_2(t)$.  To see this note that if $x_2^*(t) =
  y_2(t)$, then
  \begin{equation*}
  \ddt (x_2^*(t) - y_2(t)) = F_1(x_1^*) - F_1(y_1) \ge 0,
  \end{equation*}
  where the inequality follows since $x_1^* \ge y_1$.  Thus, we
  consider times past $\tau_2$ and redefine our initial condition to
  be the values $x(\tau_2)$ and $y(\tau_2)$.

  We note:
  \begin{align}
  \begin{split}
  \ddt (x_2^*(t) - y_2(t)) &= -\left( F_2(x_2^*) - F_2(y_2)\right) +
  F_1(x_1^*) - F_1(y_1)\\
  &= -(x_2^* - y_2)\frac{F_2(x_2^*) - F_2(y_2)}{x_2^* - y_2} + F_1(x_1^*) -
  F_1(y_1). \label{diff}
  \end{split}
  \end{align}
  To gain control over the term $F_1(x_1^*) - F_1(y_1)$ we use the
  equations governing $x_1^*$ and $y_1$:
  \begin{equation}
  x_1^*(t) - y_1(t) = x_1^*(0) - y_1(0) + \int_0^t F_1(x_1^*(s)) -
  F_1(y_1(s))ds. \label{eq}
  \end{equation}
  Rearranging and using that $x_1^* \ge y_1$, we have
  \begin{equation}
  \int_0^t F_1(x_1^*(s)) - F_1(y_1(s))ds \le x_1^*(0) - y_1(0).
  \label{etabound}
  \end{equation}
  Thus, if $\eta_1(t) = \int_0^t F_1(x_1^*(s)) - F_1(y_1(s))ds$, we have
  that $F_1(x_1^*(t)) - F_1(y_1(t)) = \eta_1'(t)$, and that for all $t$,
  $\eta_1(t) < x_1^*(0) - y_1(0)$.  Therefore,
  \begin{equation*} \ddt (x_2^*(t) - y_2(t)) = -(x_2^* - y_2)H_2(t) +
  \eta_1'(t),
  \end{equation*}
  and integrating by parts gives
  \begin{align*}
  x_2^*(t) - y_2(t) &= (x_2^*(0) - y_2(0)e^{-\int_0^tH_2(r)dr} + \int_0^t
  \eta_1'(s) e^{-\int_{s}^{t}H_2(r)dr}ds\\
  &=(x_2^*(0) - y_2(0)e^{-\int_0^tH_2(s)ds} + \eta_1(t) -
  \eta_1(0)e^{-\int_0^t H_2(r)dr}\\
  &\ \ \ \ \ - \int_0^t \eta_1(s)
  e^{-\int_{s}^{t}H_2(r)dr}H_2(s)ds.
  \end{align*}
  The last two terms are negative and, as before, the exponential
  terms go to zero as $t \to \infty$.  So, $\lim_{t \to \infty}
  |x_2^*(t) - y_2(t)| \le \eta_1(t) \le x_1^*(0) - y_1(0)$.  However, we
  can re-scale time (do the above analysis on the interval $[t/2,t]$
  instead of $[0,t]$) to conclude that $\lim_{t \to \infty} |x_2^*(t) -
  y_2(t)| \le \lim_{t \to \infty} \eta_1(t) \le \lim_{t \to \infty}
  \left( x_1^*(t/2) - y_1(t/2)\right) = 0$.

  Now we continue down the chain in a similar manner and consider
  $x_3$ and $y_3$.  Without loss of generality we may assume $x_2^*(t) >
  y_2(t)$ for all time.  If $x_3(t) < y_3(t)$ for all time, we do the
  same argument as above to conclude that $|x_3(t) - y_3(t)| \to 0$,
  as $t \to \infty$.  Thus, we assume $x_3(t) > y_3(t)$ for all time.
  The argument is the same as that above, except we now have to get
  control over $F_2(x_2^*(s)) - F_2(y_2(s))$. We have
  \begin{align*}
    x_2^*(t) - y_2(t) &= x_2^*(0) - y_2(0) + \int_0^t \left( F_1(x_1^*(s)) -
    F_1(y_1(s)) \right) ds\\
    &\ \ \ \ \ - \int_0^t \left( F_2(x_2^*(s)) -
    F_2(y_2(s)) \right) ds.
  \end{align*}
  Rearranging gives
  \begin{equation*}
    \int_0^t \left( F_2(x_2^*(s)) - F_2(y_2(s)) \right) ds = x_2^*(0) - y_2(0)
    + \eta_1(t).
  \end{equation*}
  We may define the above integral to be $\eta_2(t)$ and perform the
  same analysis as before.  In this way we continue down the chain
  and conclude that $\lim_{t \to \infty}|x(t) - y(t)| = 0$, which
  was the desired result.

\end{proof}

\begin{proof}[Proof of Theorem \ref{cor:arb_bdd}] Let $N_1^{\epsilon}
  = F_1^{-1}(I + M + \epsilon)$, where $\epsilon < K - M$.  If $x_1(t)
  > N_1^{\epsilon}$, then by the monotonicity of $F_1$ we have
  \begin{align*}
    \dot x_1(t) &= I - F(x_1(t)) + \xi(t)\\
    & \le I - F(N_1^{\epsilon}) + \xi(t)\\
    & = -M  - \epsilon + \xi(t)\\
    & < -\epsilon.
  \end{align*}
  Therefore, independent of initial conditions, $\lim_{t \to \infty}
  x_1(t) < N_1^{\epsilon}$.  However, $\epsilon$ was arbitrary, so
  $\lim_{t \to \infty} x_1(t) \le F_1^{-1}(I + M)$.  Continuing in
  this manner down the chain shows $\lim_{t \to \infty} x_i(t) \le
  F_i^{-1}(I + M)$, for each $i$. Thus, for large $t$, there exists
  $L>0$ such that $\E |x(t)| < L$.  By Chebychev's inequality we then
  have
  \begin{equation*}
    \frac{1}{t}\int_0^t P\{|x_1(s)| > R\} ds \le \frac{L}{R},
  \end{equation*}
  which converges to zero uniformly in $t$ as $R \to \infty$. As in
  the proof of Theorem \ref{cor:arb_lineargrowth}, we need to consider
  the Markov process on the extended state space $\R^n\times
  C_{ae}((-\infty,0],\R)$. As before, we obtain tightness by using the
  above estimates on the marginal of this measure in the $x(t)$
  variable since the $\xi(t)$ variable is stationary and hence already
  tight. We may again use Prohorov's Theorem to guarantee the
  existence of an invariant measure.  The proof of uniqueness is the
  same is in the proof of Theorem \ref{cor:arb_lineargrowth}.
\end{proof}

\subsection{Proof of Theorem \ref{thm:avg}}
\label{appA_avg}

\begin{proof}[Proof of Theorem \ref{thm:avg}]
  We begin by showing that $\lim_{t \to \infty} x_1(t)/t = 0$.
  Consider the dynamics governing $x_1$ where $\xi(t)$ satisfies
  \eqref{zeroavg}
\begin{equation}
  \dot x_1 = I-F_1(x_1) + \xi(t).\\
  \label{app:x1}
\end{equation}
Let $\bar x_1 = F^{-1}(I)$.  Then
\begin{equation*}
  \ddt (x_1(t) - \bar x_1) = I - F_1(x_1(t)) + \xi(t) = -\frac{I -
    F_1(x_1(t))}{\bar x_1 - x_1(t)}(x_1(t) - \bar x_1) + \xi(t). 
\end{equation*}
Setting $H(t) = \frac{I - F_1(x_1(t))}{\bar x_1 - x_1(t)} > 0$ (which
is well defined since $F_1$ is assumed differentiable and is positive
by the monotonicity of $F_1$) and using Duhamel's formula gives us
\begin{equation*}
(x_1(t) - \bar x_1) = (x_1(0) - \bar x_1)e^{-\int_0^t H(s)ds} +
\int_0^t e^{-\int_s^t H(r) dr} \xi(s) ds.
\end{equation*}
Integrating by parts gives
\begin{align*}
  (x_1(t) - \bar x_1) &= (x_1(0) - \bar x_1)e^{-\int_0^t H(s)ds} +
  \int_0^t \xi(s) ds\\
  & \ \ \ \ + \int_0^t e^{-\int_s^t H(r) dr} H(s) \left(\int_0^s
    \xi(r)dr\right) ds.
\end{align*}
By the positivity of $H(t)$ and property \eqref{zeroavg}, we then have
\begin{equation*}
  \lim_{t \to \infty} \frac{x_1(t)}{t} = \lim_{t \to \infty}
  \frac{1}{t} \int_0^t e^{-\int_s^t H(r) dr} H(s) \left(\int_0^s
    \xi(r)dr\right) ds.  
\end{equation*}
Let $\epsilon > 0$.  There exists an $S > 0$ such that $s > S$ implies
$\left|\frac{1}{s}\int_0^s \xi(r) dr \right| < \epsilon/2$.  There
exists a $T = T(s) >0$ such that $t > T$ implies $\sup_{s<S}
\left|\frac{1}{t}\int_0^s \xi(r) dr \right| < \epsilon/2$.  Therefore,
if $t > \max\{S,T\}$ we have that
\begin{align*}
  \left| \frac{1}{t} \int_0^t e^{-\int_s^t H(r) dr} \right. & \left.
    H(s) \left(\int_0^s
      \xi(r)dr\right) ds \right|\\
  &\le \int_0^t e^{-\int_s^t H(r) dr} H(s) \left|\frac{1}{t}\int_0^s
    \xi(r)dr\right|(1_{\{s \le S\}}
  + 1_{\{s > S\}}) ds\\
  & \le \epsilon \int_0^t e^{-\int_s^t H(r) dr} H(s) ds\\
  & \le \epsilon.
\end{align*}
Thus, $\lim_{t \to \infty} x_1(t)/t = 0$.

Integrating equation \eqref{app:x1}, dividing by $t$ and taking the
limit as $t \to \infty$ now gives us
\begin{equation*}
\lim_{t \to \infty} \frac{1}{t} \int_0^t F(x(s))ds = I,
\end{equation*}
which proves part 1 of Theorem \ref{thm:avg} for $x_1$.

Let $G(x) = 2\int_0^x \left(F_1(y) - I\right)dy$, which,
non-coincidentally, is the same function used in the proof of Theorem
\ref{nonlineardecrease}.  We have
\begin{equation*}
\ddt G(x_1(t)) = -2(F_1(x_1(t)) - I)^2 + 2(F_1(x_1(t)) - I)\xi(t). 
\end{equation*}
Integrating and using the inequality $ab \le (1/2) a^2 + (1/2) b^2$
gives
\begin{equation*}
  G(x_1(t)) \le G(x_1(0)) - \int_0^t (F_1(x_1(s)) - I)^2ds + \int_0^t
  \xi(s)^2ds.  
\end{equation*}
Therefore, part 2 of Theorem \ref{thm:avg} will be shown for $x_1$ if
$\liminf_{t \to \infty} G(x_1(t))/t \ge 0$.  We have
\begin{align*}
  \liminf_{t \to \infty} \frac{1}{t}G(x_1(t)) &= 2
  \liminf_{t \to \infty} \frac{1}{t} \int_0^{x_1(t)} \left(F_1(y) -
    I\right)\left(1_{\{y >
      \bar x_1\}} + 1_{\{y \le \bar x_1\}} \right)dy\\
  &\ge 2\liminf_{t \to \infty} \frac{1}{t} \int_0^{x_1(t)}
  \left(F_1(y) -
    I\right)1_{\{y \le \bar x_1\}}dy\\
  &\ge -2I \lim_{t \to \infty} \frac{x_1(t)}{t}\\
  &=0,
\end{align*}
so part 2 is shown for $x_1$. Note that parts 1 and 2 of Theorem
\ref{thm:avg} show that $F_1(x_1(t)) - I$ satisfy condition
\eqref{zeroavg}.  Therefore, to prove parts 1, 2, and 3 for all $x_i$,
one simply continues down the chain considering $F_i(x_i(t))$ as the
external perturbation of $x_{i+1}$.
\end{proof}

\subsection{Proofs of Section \ref{sec:MSCchains}}
\label{msc_proof}

The proofs of Theorems \ref{thm:msc}, \ref{msc_measurelemma},
\ref{cor:msc_arb_lineargrowth}, and \ref{cor:msc_arb_bdd} can be
handled simultaneously.

\begin{proof}[Proof of Theorems \ref{thm:msc}, \ref{msc_measurelemma},
\ref{cor:msc_arb_lineargrowth}, and \ref{cor:msc_arb_bdd}]
The key to each proof is the recognition that the species in each
complex satisfy constant multiples of the same (stochastic)
differential equations.  Mathematically this means these species can
by grouped and treated as a single substrate with a redefined
kinetics.  This reduces us to the case previously studied.  More
explicitly, there are constants $c_{ijk}$ and $d_{ijk}$ such that
$x_i^j(t) = d_{ijk}x_i^k(t) + c_{ijk}$ for all $t$.  Thus, the species
of each complex can be solved for from knowledge of just one species
from that complex. Further, a monotone increase in one translates to a
monotone increase in the others.  Choosing one species, $y_i$, from
each complex, we may redefine the $F_i$'s (and $\theta_{\delta}$ in
the white noise case) appropriately so that the vector function $y(t)$
satisfies either \eqref{sscsdes} or \eqref{sscstation} with the
$F_i$'s satisfying Assumption \ref{ssc_assump}.  Therefore, applying
the theorems of Section \ref{results} completes the proof.
\end{proof}

\section{Processes Used in Examples \ref{4-1example} and
  \ref{4-2example}}
\label{appB}

In Example \ref{4-1example} $\xi(t)$ is described as a modified
Ornstein-Uhlenbeck process such that $-10 \le \xi(t)$.  More
precisely, $\xi(t)$ is governed by the following dynamics:
\begin{equation*}
    d\xi(t)=
    \begin{cases}
       -\xi(t)dt + 4dB(t) & \text{ if }\xi(t)>-10 \\
       -\xi(t)dt &  \text{ if }\xi(t) \leq -10
    \end{cases}
\end{equation*}
This dynamics ensures that if $\xi(0)> -10$ then $\xi(t)\geq -10$ for
all $t$.

In Example \ref{4-2example}, $\xi(t)$ is built from the
Ornstein-Uhlenbeck equation $d\xi(t) = -\xi(t)dt + 3dB(t)$, with the
added condition that if $\xi(t) = -4$ or $\xi(t)=4$, then $d\xi(t) =
-\xi(t)dt$. More precisely, $\xi(t)$ is governed by the following
dynamics:
\begin{equation*}
  d\xi(t)=
  \begin{cases}
    -\xi(t)dt & \text{if } \xi(t)dt\geq 4\\
     -\xi(t)dt + 3dB(t) &\text{if } -4 < \xi(t) <4\\
     -\xi(t)dt &\text{if }  \xi(t)dt\leq -4
  \end{cases}
\end{equation*}

\bibliographystyle{amsplain}
\bibliography{AndMattRev}

\pagebreak

\noindent {\large \textbf{Figures with Captions}}

\vspace{.2in}

\noindent \textbf{Figure \ref{M-M}}

    \begin{center}
      \includegraphics[height=2.2in]{mm.eps}
    \end{center}

\vspace{.2in}

\noindent \textbf{Caption:}

The horizontal lines represent possible inputs to species $X_2$ from
species $X_1$ and the vertical lines represent the value of $x_2$ that
would give an equilibrium to the system for a given input.  We
therefore see that minor fluctuations in $F_1(x_1) = x_1$ can
correspond with large fluctuations in $x_2$.

\end{document}